\documentclass[12pt]{amsart}
\usepackage[english]{babel}
\usepackage[utf8]{inputenc}
\usepackage{lipsum}
\usepackage{mathrsfs}
\usepackage{stix}
\usepackage{fullpage}
\usepackage{mathtools}
\usepackage{amsmath}
\usepackage{tikz-cd}
\usepackage[colorlinks]{hyperref}
\usepackage{fullpage}
\usepackage{mathtools}
\usepackage{amsmath}
\usepackage{tikz-cd}
\numberwithin{equation}{section}
\theoremstyle{plain}
\newtheorem{thm}{Theorem}[section]
\newtheorem{prop}[thm]{Proposition}
\newtheorem{cor}[thm]{Corollary}
\newtheorem{lem}[thm]{Lemma}
\newtheorem{rem}[thm]{Remark}

\def\R{\mathcal R}
\def\G{\mathbb{G}}
\def\S{\mathfrak{S}}
\def\x{|x|}
\begin{document}

\title[Fractional Hardy type inequalities on homogeneous Lie groups]
{Fractional Hardy type inequalities on homogeneous Lie groups in the case $Q<sp$}

\author[A. Kassymov]{Aidyn Kassymov}
\address{
  Aidyn Kassymov:
  \endgraf
  \endgraf
  Institute of Mathematics and Mathematical Modeling
  \endgraf
  28 Shevchenko str.
  \endgraf
  050010 Almaty
  \endgraf
  Kazakhstan
  \endgraf
  and 
    \endgraf
  Department of Mathematics: Analysis, Logic and Discrete Mathematics
  \endgraf
  Ghent University, Belgium
  \endgraf
	{\it E-mail address} {\rm kassymov@math.kz}}

  \author[M. Ruzhansky]{Michael Ruzhansky}
\address{
	Michael Ruzhansky:
	 \endgraf
  Department of Mathematics: Analysis, Logic and Discrete Mathematics
  \endgraf
  Ghent University, Belgium
  \endgraf
  and
  \endgraf
  School of Mathematical Sciences
    \endgraf
    Queen Mary University of London
  \endgraf
  United Kingdom
	\endgraf
  {\it E-mail address} {\rm michael.ruzhansky@ugent.be}}

\author[D. Suragan]{Durvudkhan Suragan}
\address{
	Durvudkhan Suragan:
	\endgraf
	Department of Mathematics
	\endgraf
	School of Science and Technology, Nazarbayev University
    \endgraf
	53 Kabanbay Batyr Ave, Nur-Sultan 010000
	\endgraf
	Kazakhstan
	\endgraf
	{\it E-mail address} {\rm durvudkhan.suragan@nu.edu.kz}}

\thanks{
The first and second authors were supported in parts by the FWO Odysseus 1 grant
no. G.0H94.18N: Analysis and Partial Differential Equations, by the Methusalem programme of
the Ghent University Special Research Fund (BOF) (grant no. 01M01021) and by EPSRC grant EP/R003025/2. The first and third
authors also were supported by Ministry of Science and Higher Education grant AP23484106.
}


     \begin{abstract}
In this paper, we obtain a fractional Hardy inequality in the case $Q<sp$ on homogeneous Lie groups, and as an application we show the corresponding uncertainty principle. Also, we show a fractional Hardy-Sobolev type inequality on homogeneous Lie groups. In addition, we prove fractional logarithmic Hardy-Sobolev and fractional Nash type inequalities on homogeneous Lie groups. We note that the case $Q>sp$ was extensively studied in the literature, while here we are dealing with the complementary range $Q<sp$.
     \end{abstract}
     \maketitle
     \tableofcontents
\section{Introduction}
In the famous work \cite[Theorem B]{Har20}, G. H. Hardy showed the following discrete inequality: if $p>1$ and $\sum\limits_{n=1}^{\infty}a_{n}^{p}$ convergent, then
\begin{equation*}
    \sum_{n=1}^{\infty}\left(\frac{A_{n}}{n}\right)^{p},
\end{equation*}
is convergent, where $A_{n}=a_{1}+\ldots+a_{n}$. The integral form of aforementioned inequality is
\begin{equation}
\int_{a}^{\infty}\frac{1}{x^{p}}\left(\int_{a}^{x}f(t)dt\right)^{p}dx\leq\left(\frac{p}{p-1}\right)^{p}\int_{a}^{\infty}f^{p}(x)dx,
\end{equation}
where $f\geq0$, $p>1$, and $a>0$. At this  moment, there are many works related to  Hardy's integral inequality, see e.g. \cite{Dav99,EE04, GKPW04, KP03, KPS17} and \cite{OK90}. We  note that the multi-dimensional version of the integral Hardy inequality was proved in  \cite{DHK} for the first time.

 In \cite{RV}, the (direct) integral Hardy inequality on metric measure spaces (on homogeneous Lie groups in \cite{RY18b}) was established with applications to homogeneous Lie groups, hyperbolic spaces, Cartan-Hadamard manifolds with negative curvature and on general Lie groups with Riemannian distance for  $1<p\leq q<\infty$. 
 Let us recall this (direct) integral Hardy inequality:
 \begin{thm}[\cite{RY18b}]\label{THM:Hardy1}
Let $1<p\le q <\infty$. Let $\mathbb G $ be a homogeneous group of homogeneous dimension $Q$. 
Let $g,h$ be measurable functions  positive a.e in $\mathbb G$  such that $g\in L^1(\mathbb G\backslash \{0\})$ and $h^{1-p'}\in L^1_{loc}(\mathbb \G)$. Denote
\begin{align}
G(x):= { \int_{\mathbb{G}\backslash{B(0,|x|)}} g(y) dy}  \nonumber
\end{align} 
and 
\begin{align} 
H(x):= \int_{B(0,|x| )}h^{1-p'}(y)dy,\nonumber 
\end{align}
where $\frac{1}{p}+\frac{1}{p'}=1.$
Then the inequality
\begin{equation}\label{EQ:Hardy1}
\bigg(\int_\mathbb G\bigg(\int_{B(0,\vert x \vert)}\vert f(y) \vert dy\bigg)^q g(x)dx\bigg)^\frac{1}{q}\le C_{H}\bigg\{\int_{\mathbb G} {\vert f(x) \vert}^p h(x)dx\bigg\}^{\frac1p}
\end{equation}
hold for all measurable functions $f:\G\to{\mathbb C}$ if and only if  the following condition holds:

\begin{equation}
 D_{1} :=\sup\limits_{x\not=0} \bigg\{G^\frac{1}{q}(x) H^\frac{1}{p'}(x)\bigg\}<\infty.    
\end{equation}
Moreover, the constant $C_{H}$ for which \eqref{EQ:Hardy1} holds and the quantity $D_{1}$ are related by 
\begin{equation}\label{EQ:constants}
D_{1} \leq C_{H}\leq D_1(p')^{\frac{1}{p'}} p^\frac{1}{q}.
\end{equation}   
\end{thm}
As an example, in \cite{RY18b}, the authors showed if $g(x)=|x|^{\beta}$, $h(x)=|x|^{\alpha}$ and also,
$$\beta+Q<0,\,\,\alpha<Q(p-1),\,\,\,\text{and}\,\,\,q(\alpha+Q)-p(\beta+Q)=pqQ,$$ then \eqref{EQ:Hardy1} holds, and the constant $C_{H}$ satisfies
\begin{equation*}
    \frac{|\mathfrak{S}|^{\frac{1}{q}+\frac{1}{p'}}}{|\beta+Q|^{\frac{1}{q}}(\alpha(1-p')+Q)^{\frac{1}{p'}}}\leq C_{H}\leq (p')^{\frac{1}{p'}}p^{\frac{1}{q}}\frac{|\mathfrak{S}|^{\frac{1}{q}+\frac{1}{p'}}}{|\beta+Q|^{\frac{1}{q}}(\alpha(1-p')+Q)^{\frac{1}{p'}}},
\end{equation*}
where $|\mathfrak{S}|$ is the area of the unit sphere in $\G$ with respect to the quasi-norm $|\cdot|$. From the last formula we see that $D_{1}= \frac{|\mathfrak{S}|^{\frac{1}{q}+\frac{1}{p'}}}{|\beta+Q|^{\frac{1}{q}}(\alpha(1-p')+Q)^{\frac{1}{p'}}}.$
 In the Abelian case $\G=\mathbb{R}^{N}$, $Q=N$ and$|\cdot|=|\cdot|_{E}$ where $|\cdot|_{E}$ is the standard Euclidean distance, from \eqref{EQ:Hardy1}, we get the multi-dimensional integral Hardy inequality on Euclidean space, that is, Theorem \ref{THM:Hardy1} generalise results from \cite{DHK}. Also, we need to note that in \cite{RV21}, the authors obtained a direct integral inequality in the case $p>q$.
  In \cite{KRSin} and \cite{KRS1in}, the authors demonstrated the reverse Hardy inequality  in the cases $q<0<p<1$ and $-\infty< q\leq p<0$, respectively. 
  
   One of the well-known related inequalities of G. H. Hardy is
\begin{equation}\label{dirhar}
\left\|\frac{f}{|x|}\right\|_{L^{p}(\mathbb{R}^{N})}\leq \frac{p}{N-p}\|\nabla f\|_{L^{p}(\mathbb{R}^{N})},\,\,\,\,1<p<N,
\end{equation}
where $f\in C^{\infty}_{0}(\mathbb{R}^{N})$, $\nabla$ is the Euclidean gradient and constant $\frac{p}{N-p}$ is sharp. The integer order Hardy inequality was intensively studied, as example on the Euclidean space in \cite{Mazya85} and for manyparticles \cite{HHAT}, on the Heisenberg group in \cite{GL},  on stratified groups, for the vector fields and on homogeneous groups in \cite{RS18}.
In \cite{RS17}, the authors showed the Hardy inequality with radial derivative on homogeneous groups in the following form:
\begin{thm}[\cite{RS17}]\label{dirHardy}
Let $\G$ be a homogeneous group of homogeneous dimension $Q$. Let $|\cdot|$
be a homogeneous quasi-norm on $\G$. Let $1 < p < Q$.
Let $f\in C^{\infty}_{0}(\G\setminus\{0\})$ be a complex-valued function. Then
\begin{equation}\label{dirharhom}
\left\|\frac{f}{|x|}\right\|_{L^{p}(\G)}\leq \frac{p}{Q-p}\|\R f\|_{L^{p}(\G)},\,\,\,\,\, 1 < p < Q,
\end{equation}
where $\R=\frac{d}{d|x|}$ is the radial derivative. The constant $\frac{p}{Q-p}$ is sharp.
\end{thm}
In the Abelian case $(\mathbb{R}^{N},+)$ with $Q=N$ and $|\cdot|=|\cdot|_{E},$ where $|\cdot|_{E}$ is the standard Euclidean distance, from \eqref{dirharhom} we have
\begin{equation}\label{dirharr}
\left\|\frac{f}{|x|_{E}}\right\|_{L^{p}(\mathbb{R}^{N})}\leq \frac{p}{N-p}\left\|\frac{x}{|x|_{E}}\cdot \nabla f\right\|_{L^{p}(\mathbb{R}^{N})},\,\,\,\,\, 1 < p < N,
\end{equation}
where $\nabla$ is the standard Euclidean gradient. By using the Cauchy-Schwarz inequality from the inequality \eqref{dirharr}, we get \eqref{dirhar}.
   
   The fractional order Hardy inequality has the following form:
   \begin{equation}
       \int_{\mathbb{R}^{N}}\frac{|u(x)|^{p}}{|x|_{E}^{sp}}dx\leq C\int_{\mathbb{R}^{N}}\int_{\mathbb{R}^{N}}\frac{|u(x)-u(y)|^{p}}{|x-y|_{E}^{N+sp}}dxdy,
   \end{equation}
   where $s\in(0,1)$ and $p>1$ such that $N>sp$. The last inequality was studied in  \cite{D04} and \cite{FS08}. On stratified groups, the fractional Hardy type inequality was studied by Ciatti, Cowling and Ricci in \cite{CCR15}, on the Heisenberg group, firstly Roncal and Thangavelu proved Hardy inequality with fractional sub-Laplacian in \cite{RT15} in the case $p=2$, $s\in (0,1)$ and in \cite{AAM17}, Adimurthi and Mallick generalised it to the case $Q>sp$. Also, in \cite{KS19}, the authors generalised results of Adimurthi and  Mallick to homogeneous Lie groups settings, namely, the authors showed fractional Hardy inequality on homogeneous Lie groups in the case $Q>sp$. We especially note that in \cite{HKP97},  the authors showed the fractional Hardy inequality in the Euclidean setting, and our technique is based on that paper. In this paper, we prove the fractional Hardy inequality in the case $Q<sp$. As an application, we show the corresponding uncertainty principle on homogeneous Lie groups.  
   
   The integer order Hardy–Sobolev inequality was proved in \cite{GY00}:

\begin{equation}\label{intHSinin}
    \left(\int_{\mathbb{R}^{N}}\frac{|u(x)|^{p^{*}_{\beta}}}{|x|_{E}^{\beta}}dx\right)^{\frac{p}{p^{*}_{\beta}}}\leq C\int_{\mathbb{R}^{N}}|\nabla u(x)|^{p}dx,
\end{equation}
 where $p^{*}_{\beta}=\frac{(N-\beta)p}{N-p}$, $1<p<N$ and $\beta\in[0,p]$. In particular, the inequality \eqref{intHSinin} gives the Sobolev
inequality if $\beta = 0,$ and the $L^{p}$-Hardy inequality if $\beta = p$, i.e., \eqref{intHSinin} includes both (classical)
Hardy and Sobolev inequalities.  The analogue of \eqref{intHSinin} on graded Lie groups  was obtained in \cite{RY18b}. The fractional version of the
Hardy–Sobolev inequality was also established in \cite{Y15}:

\begin{equation}\label{intHSinfr}
    \left(\int_{\mathbb{R}^{N}}\frac{|u(x)|^{p^{*}_{s,\beta}}}{|x|_{E}^{\beta}}dx\right)^{\frac{p}{p^{*}_{s,\beta}}}\leq C\int_{\mathbb{R}^{N}}\int_{\mathbb{R}^{N}}\frac{|u(x)-u(y)|^{p}}{|x-y|_{E}^{N+sp}}dxdy,
\end{equation}
where $p^{*}_{s,\beta}= \frac{p(N -\beta)}{N - sp}, 0 \leq \beta \leq sp < N$, and $s \in (0, 1)$. From last facts, we have $1<p\leq p^{*}_{s,\beta}\leq p_{s,0}$, where $p_{s,0}=\frac{Qp}{Q-sp} $ is the fractional Sobolev exponent. Similarly, the inequality \eqref{intHSinfr}
becomes the fractional Sobolev inequality and fractional Hardy inequality, if $\beta = 0$ and   if $\beta = sp$, respectively. 
We refer to \cite{AL, BFL08, DF12, FS10} for this topic. In Euclidean settings, for $1\leq sp<N$ and $p_{s,0}=\frac{Ns}{N-sp}$ (automatically $p<p_{s,0}$) we have the fractional Sobolev inequality and for  $N=sp$ and $N<sp$ we have the  fractional  Trudinger and Morrey inequalities, respectively.   In this paper, we show the weighted fractional Hardy-Sobolev type inequality on homogeneous Lie groups. It means, we generalise the fractional Hardy inequality by the exponent parameter $1<p\leq q<\infty$ (in the particular case $Q<sp$), but we can not obtain the classical fractional Sobolev inequality. Hence, we call it the fractional Hardy-Sobolev type inequality.   

On homogeneous Lie groups, firstly the fractional logarithmic Sobolev and Hardy-Sobolev type inequalities were proved in \cite{KRS19} and \cite{KS20}, respectively. Also, another generalisation of the logarithmic Hardy-Sobolev type inequalities was proved in \cite{CKR21}. In this paper, we will establish the fractional version of the logarithmic Sobolev and Nash inequalities for the extend range of $sp>-Q$.
 \section{Main results}
 In  this  section,  we  briefly  recall  definitions and  main  properties  of  the homogeneous groups.
The comprehensive analysis on such groups has been initiated in the works of Folland and Stein \cite{FS}, but in our short exposition below we follow a more recent presentation in the open access books \cite{FR,RS18}. 

A Lie group $\mathbb{G}$ (identified with $\mathbb{R}^{N}$) 
is called a {\em homogeneous (Lie) group} when it is equipped 
with the dilation mapping
$$D_{\lambda}(x):=(\lambda^{\nu_{1}}x_{1},\ldots,\lambda^{\nu_{N}}x_{N}),\; \nu_{1},\ldots, \nu_{n}>0,\; D_{\lambda}:\mathbb{R}^{N}\rightarrow\mathbb{R}^{N},$$
which is an automorphism of the group $\mathbb{G}$ for each $\lambda>0.$
Often, one denotes $D_{\lambda}(x)$ simply by $\lambda x$.
The homogeneous dimension of a homogeneous group $\mathbb{G}$ is denoted by
\begin{equation*}
Q:=\nu_{1}+\ldots+\nu_{N}.
\end{equation*}
A homogeneous group is necessarily nilpotent and unimodular, and the 
Haar measure on $\mathbb{G}$ coincides with the Lebesgue measure (see e.g. \cite[Proposition 1.6.6]{FR}); we will denote it by $dx$. If $|S|$ denotes the volume of a measurable set $S\subset \mathbb{G}$, then
\begin{equation}\label{scal}
|D_{\lambda}(S)|=\lambda^{Q}|S| \quad {\rm and}\quad \int_{\mathbb{G}}f(\lambda x)
dx=\lambda^{-Q}\int_{\mathbb{G}}f(x)dx.
\end{equation}
A homogeneous quasi-norm on $\G$ is any continuous non-negative function
\begin{equation}
\mathbb{G}\ni x\mapsto |x|\in[0,\infty),
\end{equation}
satisfying:
\begin{itemize}
\item[a)] $|x|=|x^{-1}|$ for all $x\in\mathbb{G}$,
\item[b)] $|\lambda x|=\lambda|x|$ for all $x\in \mathbb{G}$ and $\lambda>0$,
\item[c)] $|x|=0$ if and only if $x=0$.
\end{itemize}
The quasi-ball centred at $x \in \mathbb{G}$ with radius $R > 0$ can be defined by
\begin{equation*}
B(x,R) := \{y \in\mathbb {G} : |x^{-1} y|< R\}.
\end{equation*}

\begin{prop}[\cite{FR}, Proposition 3.1.38 and \cite{RS18}, Proposition 1.2.4] \label{prop_quasi_norm1}
If $|\cdot|$ is a homogeneous quasi-norm on  $\mathbb{G}$, there exists $C_{|\cdot|}>0$ such that for every $x,y\in \mathbb{G}$, we have
\begin{equation}\label{tri}
|x y|\leq C_{|\cdot|}(|x| + |y|).
\end{equation}
\end{prop}

The following polarisation formula on homogeneous Lie groups will be used in our proofs:
there is a (unique)
positive Borel measure $\sigma$ on the
unit quasi-sphere
$
\S:=\{x\in \mathbb{G}:\,|x|=1\},
$
so that for every $f\in L^{1}(\mathbb{G})$ we have
\begin{equation}\label{EQ:polar}
\int_{\mathbb{G}}f(x)dx=\int_{0}^{\infty}
\int_{\S}f(ry)r^{Q-1}d\sigma dr.
\end{equation}
 
\subsection{Fractional Hardy inequality}
As a main result of this section, we show the fractional  Hardy inequality on homogeneous Lie groups.
\begin{thm}\label{integralhar1}
Let $\mathbb{G}$ be a homogeneous Lie group of homogeneous dimension $Q$, and let $|\cdot|$ be a quasi-norm on $\G$. Let $p>1$ be  such that $s>-\frac{Q}{p}$. Assume  that $a(x,y)$ is a non-negative measurable function in both variables on $\G\times\G$, and denote $A(x)=\left(\frac{1}{|B(0,|x|)|}\int_{B(0,|x|)}a^{1-p'}(x,y)dy\right)^{1-p}$.   Suppose that $D_{1}(p')^{\frac{1}{p'}}p^{\frac{1}{p}}<1$ where
\begin{equation}\label{D1}
D_{1}=\sup_{x\neq 0}\left[\left(\int_{\G\setminus B(0,|x|)}\frac{A(y)}{|B(0,|y|)|^{p}|y|^{sp}}dy\right)^{\frac{1}{p}}\left(\int_{B(0,|x|)}\left(\frac{A(y)}{|y|^{sp}}\right)^{1-p'}dy\right)^{\frac{1}{p'}}\right]<\infty.
\end{equation}
Then we have
\begin{equation}
    \left(\int_{\G}\frac{A(x)|u(x)|^{p}}{|x|^{sp}}dx\right)^{\frac{1}{p}}\leq \frac{\left((2C_{|\cdot|})^{-Q-sp}\frac{|\mathfrak{S}|}{Q}\right)^{-\frac{1}{p}}}{1-D_{1}(p')^{\frac{1}{p'}}p^{\frac{1}{p}}} \left(\int_{\G}\int_{\G}\frac{|u(x)-u(y)|^{p}}{|y^{-1}x|^{Q+sp}}a(x,y)dxdy\right)^{\frac{1}{p}}.
\end{equation}
\end{thm}
\begin{rem}
By rescaling $a_{\lambda}(x,y)=\lambda a(x,y)$ in \eqref{D1}, we see that $D_{1}$ is invariant under this transformation. Namely, we have $A_{\lambda}(x)=\left(\frac{\lambda^{1-p'}}{|B(0,|x|)|}\int_{B(0,|x|)}a^{1-p'}(x,y)dy\right)^{1-p}$ and then $D_{1}(\lambda)=\sup\limits_{x\neq 0}\left[\left(\int_{\G\setminus B(0,|x|)}\frac{A_{\lambda}(y)}{|B(0,|y|)|^{p}|y|^{sp}}dy\right)^{\frac{1}{p}}\left(\int_{B(0,|x|)}\left(\frac{A_{\lambda}(y)}{|y|^{sp}}\right)^{1-p'}dy\right)^{\frac{1}{p'}}\right]=D_{1}$ is independent of $\lambda$. 
\end{rem}
\begin{proof}
By using Minkowski's inequality (see Theorem 5.1.1 in \cite{Gar07}), we have
\begin{equation}\label{J1J2}
    \begin{split}
        &\left(\int_{\G}\frac{A(x)|u(x)|^{p}}{\x^{sp}}dx\right)^{\frac{1}{p}}=\Bigg(\int_{\G}\frac{A(x)}{|x|^{sp}}\Bigg{|}u(x)-\frac{1}{|B(0,\x)|}\int_{B(0,\x)}u(y)dy\\&+\frac{1}{B(0,\x)}\int_{B(0,\x)}u(y)dy\Bigg{|}^{p}dx\Bigg)^{\frac{1}{p}}\\&
        \leq \Bigg(\int_{\G}\frac{A(x)}{\x^{sp}}\Bigg{|}u(x)-\frac{1}{|B(0,\x)|}\int_{B(0,\x)}u(y)dy\Bigg{|}^{p}dx\Bigg)^{\frac{1}{p}}\\&
        +\Bigg(\int_{\G}\frac{A(x)}{|B(0,\x)|^{p}\x^{sp}}\Bigg{|}\int_{B(0,\x)}u(y)dy\Bigg{|}^{p}dx\Bigg)^{\frac{1}{p}}\\&
        =J_{1}+J_{2},
    \end{split}
\end{equation}
where 
\begin{equation}
    J_{1}=\Bigg(\int_{\G}\frac{A(x)}{\x^{sp}}\Bigg{|}u(x)-\frac{1}{|B(0,\x)|}\int_{B(0,\x)}u(y)dy\Bigg{|}^{p}dx\Bigg)^{\frac{1}{p}},
\end{equation}
and 
\begin{equation}
    J_{2}=\Bigg(\int_{\G}\frac{A(x)}{\x^{sp}|B(0,\x)|^{p}}\Bigg{|}\int_{B(0,\x)}u(y)dy\Bigg{|}^{p}dx\Bigg)^{\frac{1}{p}}.
\end{equation}
First, let us show that 
$$J_{1}\leq \left((2C_{|\cdot|})^{-Q-sp}\frac{|\mathfrak{S}|}{Q}\right)^{-\frac{1}{p}}\left(\int_{\G}\int_{\G}\frac{|u(x)-u(y)|^{p}}{|y^{-1}x|^{Q+sp}}a(x,y)dxdy\right)^{\frac{1}{p}}.$$
By using H\"{o}lder's inequality with $\frac{1}{p}+\frac{1}{p'}=1$, we have
\begin{multline}
    \left|\int_{B(0,\x)}(u(x)-u(y))dy\right|\\
    \leq \left(\int_{B(0,\x)}|u(x)-u(y)|^{p}a(x,y)dy\right)^{\frac{1}{p}}\left(\int_{B(0,|x|)}a^{1-p'}(x,y)dy\right)^{\frac{p-1}{p}}.
\end{multline}
By using this fact  with $|y^{-1}x|\leq C_{|\cdot|}(\x+|y|)\leq 2C_{|\cdot|}\x$ for all $y\in B(0,|x|)$ (see Proposition \ref{prop_quasi_norm1}) and $|B(0,|x|)|=\frac{|\mathfrak{S}||x|^{Q}}{Q}$, we have 
\begin{equation*}
    \begin{split}
    &\int_{\G}\int_{\G}\frac{|u(x)-u(y)|^{p}}{|y^{-1}x|^{Q+sp}}a(x,y)dxdy\\&
    \geq \int_{\G}\int_{B(0,\x)}\frac{|u(x)-u(y)|^{p}}{|y^{-1}x|^{Q+sp}}a(x,y)dxdy\\&
    \geq (2C_{|\cdot|})^{-Q-sp}\int_{\G}\frac{1}{|x|^{Q+sp}}dx\int_{B(0,|x|)}|u(x)-u(y)|^{p}a(x,y)dy\\&
    \geq (2C_{|\cdot|})^{-Q-sp}\int_{\G}\frac{\left(\int_{B(0,|x|)}a^{1-p'}(x,y)dy\right)^{1-p}}{|x|^{Q+sp}}dx\left|\int_{B(0,\x)}(u(x)-u(y))dy\right|^{p}dx\\&
    =(2C_{|\cdot|})^{-Q-sp}\int_{\G}\frac{\left(\int_{B(0,|x|)}a^{1-p'}(x,y)dy\right)^{1-p}}{|x|^{Q+sp}}dx\left|u(x)|B(0,|x|)|-\int_{B(0,\x)}u(y)dy\right|^{p}dx\\&
    =(2C_{|\cdot|})^{-Q-sp}\frac{|\mathfrak{S}|}{Q}\int_{\G}\frac{A(x)}{\x^{sp}}\left|u(x)-\frac{1}{|B(0,\x)|}\int_{B(0,\x)}u(y)dy\right|^{p}dx\\&
    =(2C_{|\cdot|})^{-Q-sp}\frac{|\mathfrak{S}|}{Q}J_{1}^{p}.
    \end{split}
\end{equation*}
Finally, we obtain 
\begin{equation}\label{J1}
    J^{p}_{1}\leq \left((2C_{|\cdot|})^{-Q-sp}\frac{|\mathfrak{S}|}{Q}\right)^{-1}\int_{\G}\int_{\G}\frac{|u(x)-u(y)|^{p}}{|y^{-1}x|^{Q+sp}}a(x,y)dxdy.
\end{equation}
By assumption, we have that 
\begin{equation}\label{ass123}
D_{1}=\sup_{x\neq 0}\left[\left(\int_{\G\setminus B(0,|x|)}\frac{A(y)}{|B(0,|y|)|^{p}|y|^{sp}}dy\right)^{\frac{1}{p}}\left(\int_{B(0,|x|)}\left(\frac{A(y)}{|y|^{sp}}\right)^{1-p'}dy\right)^{\frac{1}{p'}}\right]<\infty.
\end{equation}
By taking $g(x)=\frac{A(x)}{|x|^{sp}|B(0,|x|)|^{p}}$, $h(x)=\frac{A(x)}{|x|^{sp}}$, $f=u$ and $p=q$ in Theorem \ref{THM:Hardy1} and combining the last fact with \eqref{ass123}  and \eqref{EQ:constants}, we have  
\begin{equation}
    J_{2}\leq C_{H}\left(\int_{\G}\frac{A(x)|u(x)|^{p}}{|x|^{sp}}dx\right)^{\frac{1}{p}}
    \leq D_{1}(p')^{\frac{1}{p'}}p^{\frac{1}{p}}\left(\int_{\G}\frac{A(x)|u(x)|^{p}}{|x|^{sp}}dx\right)^{\frac{1}{p}}.
\end{equation}
Putting  \eqref{J1} and last fact with $D_{1}(p')^{\frac{1}{p'}}p^{\frac{1}{p}}<1$ in \eqref{J1J2}, we get
\begin{equation*}
   \left(\int_{\G}\frac{A(x)|u(x)|^{p}}{|x|^{sp}}dx\right)^{\frac{1}{p}}\leq \frac{\left((2C_{|\cdot|})^{-Q-sp}\frac{|\mathfrak{S}|}{Q}\right)^{-\frac{1}{p}}}{1-D_{1}(p')^{\frac{1}{p'}}p^{\frac{1}{p}}} \left(\int_{\G}\int_{\G}\frac{|u(x)-u(y)|^{p}}{|y^{-1}x|^{Q+sp}}a(x,y)dxdy\right)^{\frac{1}{p}},
\end{equation*}
completing the proof.
\end{proof}
\begin{cor}\label{cor1}
Under the same assumptions in Theorem \ref{integralhar1} with $a(x,y)=1$, $s>0$ and $p>1$ such that $sp>Q$, we have 
\begin{equation}\label{fracharpar}
     \left(\int_{\G}\frac{|u(x)|^{p}}{|x|^{sp}}dx\right)^{\frac{1}{p}}\leq \frac{\left((2C_{|\cdot|})^{-Q-sp}\frac{|\mathfrak{S}|}{Q}\right)^{-\frac{1}{p}}}{1-\frac{p}{sp+Qp-Q}}\left(\int_{\G}\int_{\G}\frac{|u(x)-u(y)|^{p}}{|y^{-1}x|^{Q+sp}}dxdy\right)^{\frac{1}{p}}.
\end{equation}

\end{cor}
\begin{proof}
Let us check the condition \eqref{D1} with $A(x)=a(x,y)=1$. By noting that the $|B(0,|x|)|=\frac{|\mathfrak{S}||x|^{Q}}{Q}$ and by using $Q<sp$, we have 
\begin{equation*}
\begin{split}
 \left(\int_{\G\setminus B(0,|x|)}\frac{1}{|B(0,|y|)|^{p}|y|^{sp}}dy\right)^{\frac{1}{p}}\left(\int_{B(0,|x|)}|y|^{sp'}dy\right)^{\frac{1}{p'}}&=\left(\frac{Q^{p}|\mathfrak{S}|^{-p+1}|x|^{-sp-Qp+Q}}{sp+Qp-Q}\right)^{\frac{1}{p}}
 \left(\frac{|\mathfrak{S}||x|^{sp'+Q}}{sp'+Q}\right)^{\frac{1}{p'}}\\&
 =\frac{Q|\mathfrak{S}|^{-1+\frac{1}{p}+\frac{1}{p'}}|x|^{-s-Q+\frac{Q}{p}+s+\frac{Q}{p'}}}{(sp+Qp-Q)^{\frac{1}{p}}(sp'+Q)^{\frac{1}{p'}}}\\&
 =\frac{Q}{(sp+Qp-Q)^{\frac{1}{p}}\left(s\frac{p}{p-1}+Q\right)^{\frac{1}{p'}}}\\&
 =\frac{Q(p-1)^{\frac{1}{p'}}}{sp+Qp-Q},
\end{split}
\end{equation*}
that is,
\begin{equation}
\begin{split}
    D_{1}&=\sup_{x\neq 0}\left[\left(\int_{\G\setminus B(0,|x|)}\frac{1}{|B(0,|y|)|^{p}|y|^{sp}}dy\right)^{\frac{1}{p}}\left(\int_{B(0,|x|)}\left(\frac{1}{|y|^{sp}}\right)^{1-p'}dy\right)^{\frac{1}{p'}}\right]\\&
    =\frac{Q(p-1)^{\frac{1}{p'}}}{sp+Qp-Q}<\infty.
\end{split}
\end{equation}
Finally, by using $sp>Q$ we have 
\begin{equation}
  D_{1}(p')^{\frac{1}{p'}}p^{\frac{1}{p}}=\frac{Q(p-1)^{\frac{1}{p'}}}{sp+Qp-Q}\frac{p}{(p-1)^{\frac{1}{p'}}}\stackrel{sp>Q}< \frac{Qp}{Qp}=1,
\end{equation}
so that Theorem \ref{integralhar1} can be applied.
\end{proof}
\begin{rem}
In Corollary \ref{cor1}, we have  condition the $Q<sp$, which is stronger than condition $s>-\frac{Q}{p}$ in Theorem \ref{integralhar1}. 
\end{rem}

\begin{cor}[Uncertainty principle]
Let $\G$ be homogeneous Lie group with homogeneous dimension $Q$,  and let $p>1$ and $s>0$ be such that $Q<sp$. Then we have
\begin{equation}
    \int_{\G}|u(x)|^{2}dx\leq  C\left(\int_{\G}\int_{\G}\frac{|u(x)-u(y)|^{p}}{|y^{-1}x|^{Q+sp}}dxdy\right)^{\frac{1}{p}} \left(\int_{\G}|x|^{sp'}|u(x)|^{p'}dx\right)^{\frac{1}{p'}}.
\end{equation}
\end{cor}
\begin{proof}
By using H\"{o}lder's inequality and \eqref{fracharpar}, we get
\begin{equation*}
    \begin{split}
        \int_{\G}|u(x)|^{2}dx&=\int_{\G}|x|^{s}|u(x)||x|^{-s}|u(x)|dx\leq \left(\int_{\G}|x|^{-sp}|u(x)|^{p}dx\right)^{\frac{1}{p}}\left(\int_{\G}|x|^{sp'}|u(x)|^{p'}dx\right)^{\frac{1}{p'}}\\&
        \leq C\left(\int_{\G}\int_{\G}\frac{|u(x)-u(y)|^{p}}{|y^{-1}x|^{Q+sp}}dxdy\right)^{\frac{1}{p}} \left(\int_{\G}|x|^{sp'}|u(x)|^{p'}dx\right)^{\frac{1}{p'}},
    \end{split}
\end{equation*}
completing the proof.
\end{proof}

\subsection{Fractional Hardy-Sobolev type inequality}
In this subsection, we show fractional Hardy-Sobolev inequality.
\begin{thm}\label{HSin}
Let $\mathbb{G}$ be  a homogeneous Lie groups with homogeneous dimension $Q$, and let $|\cdot|$ be a quasi-norm on $\G$. Let $1<p\leq q<\infty$ be such that $s>-\frac{Q}{p}$.  Assume  that $v(x)$ and $z(x)$ are  non-negative measurable functions on $\G$ such that 
\begin{equation}\label{A1}
    C(|x|):=\int_{B(0,|x|)}v(y)dy,
\end{equation}
and
\begin{equation}\label{A2}
    A(x):=\left(\frac{1}{|B(0,|x|)|}\int_{B(0,|x|)}\left(\frac{v^{p}(y)}{z(y)}\right)^{\frac{1}{p-1}}dy\right)^{-\frac{q}{p'}}\left(\frac{C(|x|)}{|B(0,|x|)|}\right)^{q}v(x).
\end{equation}
Suppose that $D_{1}(q')^{\frac{1}{q'}} q^\frac{1}{q}<1$ where
\begin{equation}\label{D1HS1}
D_{1}=\sup_{x\neq 0}\left[\left(\int_{\G\setminus B(0,|x|)}\frac{A(y)C^{-q}(|y|)}{|y|^{sq}}dy\right)^{\frac{1}{q}}\left(\int_{B(0,|x|)}\left(\frac{A(y)v^{-q}(y)}{|y|^{sq}}\right)^{1-q'}dy\right)^{\frac{1}{q'}}\right]<\infty.
\end{equation}
Then we have
\begin{equation}
    \left(\int_{\G}\frac{A(x)|u(x)|^{q}}{|x|^{sq}}dx\right)^{\frac{1}{q}}\leq \frac{(2C_{|\cdot|})^{\frac{Q+sp}{p}}\left(\frac{|\mathfrak{S}|}{Q}\right)^{-\frac{1}{p}}}{1-D_{1}(q')^{\frac{1}{q'}}q^{\frac{1}{q}}}\left(\int_{\G}\left(\int_{\G}\frac{|u(x)-u(y)|^{p}}{|y^{-1}x|^{Q+sp}}z(x)dx\right)^{\frac{q}{p}}v(y)dy\right)^{\frac{1}{q}}.
\end{equation}
\end{thm}
\begin{rem}
If $p=q$, $v=1$ (that is, we have $a(x,y)=z(y)$ in Theorem \ref{integralhar1}) and under the same assumption as in Theorem \ref{HSin}, we obtain the fractional Hardy inequality
\begin{equation}
     \left(\int_{\G}\frac{A(x)|u(x)|^{p}}{|x|^{sp}}dx\right)^{\frac{1}{p}}\leq C\left(\int_{\G}\int_{\G}\frac{|u(x)-u(y)|^{p}}{|y^{-1}x|^{Q+sp}}z(x)dxdy\right)^{\frac{1}{p}}.
\end{equation}
\end{rem}
\begin{proof}[Proof of Theorem \ref{HSin}]
The strategy of the proof is similar to  Theorem \ref{integralhar1}. By using H\"older's inequality, we have
\begin{multline}\label{ASD}
    \left(\int_{B(0,|y|)}|u(x)-u(y)|v(x)dx\right)^{q}\\
    \leq \left(\int_{B(0,|y|)}|u(x)-u(y)|^{p}z(x)dx\right)^{\frac{q}{p}}\left(\int_{B(0,|y|)}\left(\frac{v^{p}(x)}{z(x)}\right)^{\frac{1}{p-1}}dx\right)^{\frac{q}{p'}}.
\end{multline}
By noting that the $|B(0,\x)|=\frac{|\S||x|^{Q}}{Q}$, we get 
\begin{equation}\label{A123}
    \begin{split}
        \left(\int_{B(0,|y|)}\left(\frac{v^{p}(x)}{z(x)}\right)^{\frac{1}{p-1}}dx\right)^{-\frac{q}{p'}}\frac{(C(|y|))^{q}v(y)}{|y|^{(Q+sp)\frac{q}{p}}}&=\left(\frac{|B(0,|y|)|}{|B(0,|y|)|}\int_{B(0,|y|)}\left(\frac{v^{p}(x)}{z(x)}\right)^{\frac{1}{p-1}}dx\right)^{-\frac{q}{p'}}\\&
        \times\left(\frac{|B(0,|y|)|C(|y|)}{|B(0,|y|)|}\right)^{q}\frac{v(y)}{|y|^{(Q+sp)\frac{q}{p}}}\\&
        \stackrel{\eqref{A2}}=A(y)\frac{|B(0,|y|)|^{\frac{q}{p}}}{|y|^{(Q+sp)\frac{q}{p}}}\\&
        =\frac{|\S|^{\frac{q}{p}}}{Q^{\frac{q}{p}}|y|^{sq}}A(y)
    \end{split}
\end{equation}
and by using  the last fact with \eqref{ASD} and $|y^{-1}x|\leq 2C_{|\cdot|}|y|$ for all $x\in B(0,|y|)$, we have\begin{equation*}
    \begin{split}
     &\int_{\G}\left(\int_{\G}\frac{|u(x)-u(y)|^{p}}{|y^{-1}x|^{Q+sp}}z(x)dx\right)^{\frac{q}{p}}v(y)dy\geq \int_{\G}\left(\int_{B(0,|y|)}\frac{|u(x)-u(y)|^{p}}{|y^{-1}x|^{Q+sp}}z(x)dx\right)^{\frac{q}{p}}v(y)dy \\&
     \geq (2C_{|\cdot|})^{(-Q-sp)\frac{q}{p}}\int_{\G}\left(\int_{B(0,|y|)}|u(x)-u(y)|^{p}z(x)dx\right)^{\frac{q}{p}}\frac{v(y)}{|y|^{(Q+sp)\frac{q}{p}}}dy\\&
     \geq (2C_{|\cdot|})^{(-Q-sp)\frac{q}{p}}\int_{\G}\left(\int_{B(0,|y|)}|u(x)-u(y)|v(x)dx\right)^{q}\left(\int_{B(0,|y|)}\left(\frac{v^{p}(x)}{z(x)}\right)^{\frac{1}{p-1}}dx\right)^{-\frac{q}{p'}}\frac{v(y)}{|y|^{(Q+sp)\frac{q}{p}}}dy\\&
     \geq (2C_{|\cdot|})^{(-Q-sp)\frac{q}{p}}\int_{\G}\left|u(y)\int_{B(0,|y|)}v(x)dx-\int_{B(0,|y|)}u(x)v(x)dx\right|^{q}\left(\int_{B(0,|y|)}\left(\frac{v^{p}(x)}{z(x)}\right)^{\frac{1}{p-1}}dx\right)^{-\frac{q}{p'}}\frac{v(y)}{|y|^{(Q+sp)\frac{q}{p}}}dy\\&
     \stackrel{\eqref{A1}}=(2C_{|\cdot|})^{(-Q-sp)\frac{q}{p}}\int_{\G}\left|u(y)C(|y|)-\int_{B(0,|y|)}u(x)v(x)dx\right|^{q}\left(\int_{B(0,|y|)}\left(\frac{v^{p}(x)}{z(x)}\right)^{\frac{1}{p-1}}dx\right)^{-\frac{q}{p'}}\frac{v(y)}{|y|^{(Q+sp)\frac{q}{p}}}dy\\&
     \stackrel{\eqref{A123}}=(2C_{|\cdot|})^{(-Q-sp)\frac{q}{p}}\left(\frac{|\mathfrak{S}|}{Q}\right)^{\frac{q}{p}}\int_{\G}\left|u(y)-\frac{1}{C(|y|)}\int_{B(0,|y|)}u(x)v(x)dx\right|^{q}\frac{A(y)}{|y|^{sq}}dy\\&
     =(2C_{|\cdot|})^{(-Q-sp)\frac{q}{p}}\left(\frac{|\mathfrak{S}|}{Q}\right)^{\frac{q}{p}}J^{q}_{1},
    \end{split}
\end{equation*}
where 
$$J_{1}:=\left(\int_{\G}\left|u(y)-\frac{1}{C(|y|)}\int_{B(0,|y|)}u(x)v(x)dx\right|^{q}\frac{A(y)}{|y|^{sq}}dy\right)^{\frac{1}{q}}.$$
By using \eqref{D1HS1} and \eqref{EQ:Hardy1} (with $p=q$, $f=uv$ and weights $g(x)=\frac{A(x)}{C^{q}(|x|)|x|^{sq}}$ and $h(x)=\frac{A(x)}{\x^{sq}v^{q}(x)}$), we have
\begin{equation*}
\begin{split}
    J_{2}:&=\left(\int_{\G}\frac{A(y)C^{-q}(|y|)}{|y|^{sq}}\left(\int_{B(0,|y|)}u(x)v(x)dx\right)^{q}dy\right)^{\frac{1}{q}}\leq D_{1}(q')^{\frac{1}{q'}}q^{\frac{1}{q}} \left(\int_{\G}\frac{A(x)|u(x)|^{q}}{|x|^{sq}}dx\right)^{\frac{1}{q}}.
\end{split}
\end{equation*}
Hence, by combining last two facts with Minkowski's inequality, we obtain
\begin{equation*}
    \begin{split}
        & \left(\int_{\G}\frac{A(y)|u(y)|^{q}}{|y|^{sq}}dy\right)^{\frac{1}{q}}=\Bigg(\int_{\G}\frac{A(y)}{|y|^{sq}}\Bigg{|}u(y)-\frac{1}{C(|y|)}\int_{B(0,|y|)}u(x)v(x)dx\\&
        +\frac{1}{C(|y|)}\int_{B(0,|y|)}u(x)v(x)dx\Bigg{|}^{q}dy\Bigg)^{\frac{1}{q}}\\&
        \leq \Bigg(\int_{\G}\frac{A(y)}{|y|^{sq}}\Bigg{|}u(y)-\frac{1}{C(|y|)}\int_{B(0,|y|)}u(x)v(x)dx\Bigg{|}^{q}dy\Bigg)^{\frac{1}{q}}
        +\Bigg(\int_{\G}\frac{A(y)}{C^{q}(|y|)|y|^{sq}}\Bigg{|}\int_{B(0,|y|)}u(x)v(x)dx\Bigg{|}^{q}dy\Bigg)^{\frac{1}{q}}\\&
        =J_{1}+J_{2}\\&
        \leq (2C_{|\cdot|})^{\frac{Q+sp}{p}}\left(\frac{|\mathfrak{S}|}{Q}\right)^{-\frac{1}{p}}\left(\int_{\G}\left(\int_{\G}\frac{|u(x)-u(y)|^{p}}{|y^{-1}x|^{Q+sp}}z(x)dx\right)^{\frac{q}{p}}v(y)dy\right)^{\frac{1}{q}}+ D_{1}(q')^{\frac{1}{q'}}q^{\frac{1}{q}} \left(\int_{\G}\frac{A(y)|u(y)|^{q}}{|y|^{sq}}dy\right)^{\frac{1}{q}},
    \end{split}
\end{equation*}
it means by using $D_{1}(p')^{\frac{1}{q'}}p^{\frac{1}{q}}<1$, we get
\begin{equation*}
    \left(\int_{\G}\frac{A(x)|u(x)|^{q}}{|x|^{sq}}dx\right)^{\frac{1}{q}}\leq \frac{(2C_{|\cdot|})^{\frac{Q+sp}{p}}\left(\frac{|\mathfrak{S}|}{Q}\right)^{-\frac{1}{p}}}{1-D_{1}(q')^{\frac{1}{q'}}q^{\frac{1}{q}}}\left(\int_{\G}\left(\int_{\G}\frac{|u(x)-u(y)|^{p}}{|y^{-1}x|^{Q+sp}}z(x)dx\right)^{\frac{q}{p}}v(y)dy\right)^{\frac{1}{q}},
\end{equation*}
completing the proof.
\end{proof}
\subsection{Logarithmic Hardy-Sobolev type inequality}
Firstly, let us recall logarithmic H\"older's inequality on homogeneous Lie groups.
\begin{lem}[Logarithmic H\"older inequality, \cite{KRS19,CKR21}]\label{holder}
Let $\G$ be a homogeneous Lie group. Let $u\in L^{p}(\mathbb{G})\cap L^{q}(\mathbb{G})\setminus\{0\}$ with some $1<p<q< \infty.$ 
Then we have
\begin{equation}\label{holdernn}
\int_{\mathbb{G}}\frac{|u|^{p}}{\|u\|^{p}_{L^{p}(\mathbb{G})}}\log\left(\frac{|u|^{p}}{\|u\|^{p}_{L^{p}(\mathbb{G})}}\right)dx\leq \frac{q}{q-p}\log\left(\frac{\|u\|^{p}_{L^{q}(\mathbb{G})}}{\|u\|^{p}_{L^{p}(\mathbb{G})}}\right).
\end{equation}
\end{lem}
Let us now show the logarithmic fractional Hardy-Sobolev type inequality on homogeneous Lie groups.
\begin{thm}\label{logHS}
Let $\mathbb{G}$ be a homogeneous Lie group of homogeneous dimension $Q$ and let $|\cdot|$ be a quasi-norm on $\G$. Assume  that $v(x)$ and $z(x)$ are  non-negative measurable functions on $\G$. Let $1<p< q<\infty$ be such that $s>-\frac{Q}{p}$.  Suppose that $D_{1}(p')^{\frac{1}{p'}}(p)^{\frac{1}{q}}<1$ where
\begin{equation}\label{D1HS12}
D_{1}=\sup_{x\neq 0}\left[\left(\int_{\G\setminus B(0,|x|)}\frac{A(y)C^{-q}(|y|)}{|y|^{sq}}dy\right)^{\frac{1}{q}}\left(\int_{B(0,|x|)}\left(\frac{A(y)v^{-q}(y)}{|y|^{sq}}\right)^{1-q'}dy\right)^{\frac{1}{q'}}\right]<\infty,
\end{equation}
\begin{equation}\label{A3}
    C(|x|):=\int_{B(0,|x|)}v(y)dy,
\end{equation}
and
\begin{equation}\label{A4}
    A(x):=\left(\frac{1}{|B(0,|x|)|}\int_{B(0,|x|)}\left(\frac{v^{p}(y)}{z(y)}\right)^{\frac{1}{p-1}}dy\right)^{-\frac{q}{p'}}\left(\frac{C(|x|)}{|B(0,|x|)|}\right)^{q}v(x).
\end{equation}
Then we have
\begin{equation}
    \int_{\mathbb{G}}\frac{\left(\frac{A^{\frac{1}{q}}|u|}{|x|^{s}}\right)^{p}}{\left\|\frac{A^{\frac{1}{q}}u}{|x|^{s}}\right\|^{p}_{L^{p}(\mathbb{G})}}\log\left(\frac{\left(\frac{A^{\frac{1}{q}}|u|}{|x|^{s}}\right)^{p}}{\left\|\frac{A^{\frac{1}{q}}u}{|x|^{s}}\right\|^{p}_{L^{p}(\mathbb{G})}}\right)dx\leq \frac{q}{q-p}\log\left(C\frac{\left(\int_{\G}\left(\int_{\G}\frac{|u(x)-u(y)|^{p}}{|y^{-1}x|^{Q+sp}}z(x)dx\right)^{\frac{q}{p}}v(y)dy\right)^{\frac{p}{q}}}{\left\|\frac{A^{\frac{1}{q}}u}{|x|^{s}}\right\|^{p}_{L^{p}(\mathbb{G})}}\right),
\end{equation}
where $C$ is a positive constant independent of $u$.
\end{thm}
\begin{proof}
By combining the logarithmic H\"older inequality and the fractional Hardy-Sobolev type inequality, we get
\begin{equation}
    \begin{split}
        &\int_{\mathbb{G}}\frac{\left(\frac{A^{\frac{1}{q}}u}{|x|^{s}}\right)^{p}}{\left\|\frac{A^{\frac{1}{q}}u}{|x|^{s}}\right\|^{p}_{L^{p}(\mathbb{G})}}\log\left(\frac{\left(\frac{A^{\frac{1}{p}}|u|}{|x|^{s}}\right)^{p}}{\left\|\frac{A^{\frac{1}{q}}u}{|x|^{s}}\right\|^{p}_{L^{p}(\mathbb{G})}}\right)dx\leq \frac{q}{q-p} \log\left(\frac{\left\|\frac{A^{\frac{1}{q}}u}{|x|^{s}}\right\|^{p}_{L^{q}(\mathbb{G})}}{\left\|\frac{A^{\frac{1}{q}}u}{|x|^{s}}\right\|^{p}_{L^{p}(\mathbb{G})}}\right)\\&
        \leq \frac{q}{q-p}\log\left(C\frac{\left(\int_{\G}\left(\int_{\G}\frac{|u(x)-u(y)|^{p}}{|y^{-1}x|^{Q+sp}}z(x)dx\right)^{\frac{q}{p}}v(y)dy\right)^{\frac{p}{q}}}{\left\|\frac{A^{\frac{1}{q}}u}{|x|^{s}}\right\|^{p}_{L^{p}(\mathbb{G})}}\right),
    \end{split}
\end{equation}
completing the proof.
\end{proof}
\subsection{Fractional Nash type inequality}
Finally, we show the fractional Nash type inequality on homogeneous Lie groups.
\begin{thm}\label{THM:Nash}
Let $\G$ be a homogeneous group with homogeneous dimension $Q$and let $|\cdot|$ be a quasi-norm on $\G$. Let $2< q<\infty$ be  such that $s>-\frac{Q}{2}$.  Assume  that $v(x)$ and $z(x)$ are  non-negative measurable functions on $\G$. Suppose that $2^{\frac{1}{2}+\frac{1}{q}}D_{1}<1$ where
\begin{equation}\label{D1HS3}
D_{1}=\sup_{x\neq 0}\left[\left(\int_{\G\setminus B(0,|x|)}\frac{A(y)C^{-q}(|y|)}{|y|^{sq}}dy\right)^{\frac{1}{q}}\left(\int_{B(0,|x|)}\left(\frac{A(y)v^{-q}(y)}{|y|^{sq}}\right)^{1-q'}dy\right)^{\frac{1}{q'}}\right]<\infty,
\end{equation}
\begin{equation}
    C(|x|):=\int_{B(0,|x|)}v(y)dy,
\end{equation}
and
\begin{equation}
    A(x):=\left(\frac{1}{|B(0,|x|)|}\int_{B(0,|x|)}\frac{v^{2}(y)}{z(y)}dy\right)^{-\frac{q}{2}}\left(\frac{C(|x|)}{|B(0,|x|)|}\right)^{q}v(x).
\end{equation} Then, the fractional Nash type inequality is given by 
 \begin{equation}\label{Nash-final2t}
    \left\|\frac{A^{\frac{1}{q}}}{|\cdot|^{s}}u\right\|_{L^2(\G)}^{4-\frac{4}{q}} \leq C \left(\int_{\G}\left(\int_{\G}\frac{|u(x)-u(y)|^{2}}{|y^{-1}x|^{Q+2s}}z(x)dx\right)^{\frac{q}{2}}v(y)dy\right)^{\frac{2}{q}}\left\|\frac{A^{\frac{1}{q}}}{|\cdot|^{s}}u\right\|_{L^1(\G)}^{\frac{2(q-2)}{q}},
\end{equation}
 where $C>0$ is independent of $u$.
\end{thm}
\begin{proof}
By choosing $q>2$ and $g(x)=A^{\frac{1}{q}}(x)|x|^{-s}u(x)$, note that  $|g|^2 \|g\|^{-2}_{L^2(\G)}\,dx$ is a probability measure on $\G$. By using Jensen's inequality  for $h(u)=\log \frac{1}{u}$,  we have 
 \begin{eqnarray*}
\log \left(\frac{\|g\|^{2}_{L^2(\G)}}{\|g\|_{L^1(\G)}} \right) & = & h \left( \frac{\|g\|_{L^1(\G)}}{\|g\|^{2}_{L^2(\G)}}\right)=h \left( \int_{\G}\frac{1}{|g|}|g|^2 \|g\|^{-2}_{L^2(\G)}\,dx\right)\\
& \leq & \int_{\G} h \left( \frac{1}{|g|}\right) |g|^2 \|g\|^{-2}_{L^2(\G)}\,dx=\int_{\G} \frac{|g|^2}{\|g\|^{2}_{L^2(\G)}}\log |g| \,dx\,.
 \end{eqnarray*}
 Then by using logarithmic Hardy-Sobolev inequality for $p=2$ in Theorem \ref{logHS},  we obtain 
 \[
 \int_{\G} \frac{|g|^2}{\|g\|^{2}_{L^2(\G)}}\log |g| \,dx \leq \log \|g\|_{L^2(\G)}+\frac{q}{2(q-2)} \log \left( C \frac{\left(\int_{\G}\left(\int_{\G}\frac{|u(x)-u(y)|^{2}}{|y^{-1}x|^{Q+2s}}z(x)dx\right)^{\frac{q}{2}}v(y)dy\right)^{\frac{2}{q}}}{\|g\|^{2}_{L^2(\G)}}\right)\,,
 \]
  hence, we have
 \[
 \left(1+\frac{q}{q-2} \right) \log \|g\|_{L^2(\G)} \leq \frac{q}{2(q-2)} \log \left(C \|g\|_{L^1(\G)}^{\frac{2(q-2)}{q}}\left(\int_{\G}\left(\int_{\G}\frac{|u(x)-u(y)|^{2}}{|y^{-1}x|^{Q+2s}}z(x)dx\right)^{\frac{q}{2}}v(y)dy\right)^{\frac{2}{q}} \right)\,,
 \]
 which is equivalent to 
 \begin{equation}\label{Nash-final2}
 \left\|\frac{A^{\frac{1}{q}}}{|\cdot|^{s}}u\right\|_{L^2(\G)}^{4-\frac{4}{q}} \leq C \left(\int_{\G}\left(\int_{\G}\frac{|u(x)-u(y)|^{2}}{|y^{-1}x|^{Q+2s}}z(x)dx\right)^{\frac{q}{2}}v(y)dy\right)^{\frac{2}{q}}\left\|\frac{A^{\frac{1}{q}}}{|\cdot|^{s}}u\right\|_{L^1(\G)}^{\frac{2(q-2)}{q}},
 \end{equation}
completing the proof.
\end{proof}


\begin{thebibliography}{H8}

\bibitem{AAM17}
Adimurthi and A. Mallick. A Hardy type inequality on fractional order Sobolev spaces on the Heisenberg group. {\em Ann. Scuola Norm. Super. Pisa-Cl. Sci.}, 18(3):917--949, 2018.
\bibitem{AL}
F. Avkhadiev and A. Laptev. Hardy inequalities for nonconvex domains. {\it In: Around the
Research of Vladimir Maz’ya. I. Function Spaces}, pp. 1–12, Springer, New York (2010).
\bibitem{B15}
W. Beckner. Functionals for multilinear fractional embedding. {\em Acta Math. Sin. (Engl.Ser.,)} 31(1):1--28, DOI 10.1007/s10114-015-4321-6, 2015.
\bibitem{BFL08}
R. D. Benguria, R. L. Frank, and M. Loss. The sharp constant in the Hardy–Sobolev–
Maz’ya inequality in the three dimensional upper half-space. {\it Math. Res. Lett.}, 15(4): 613--622 2008.
\bibitem{CKR21}
M. Chatzakou, A. Kassymov and M. Ruzhansky. Logarithmic Sobolev-type inequalities on Lie groups. {\em arXiv:2106.15652}, 2021. To appear in {\em J. Geom. Anal.}
\bibitem{CCR15}
P. Ciatti, M. G. Cowling and F. Ricci. Hardy and uncertainty inequalities on stratified Lie groups. {\it Adv. Math.}, 277:365--387, 2015.
\bibitem{Dav99}
E.~B. Davies.
\newblock A review of {H}ardy inequalities.
\newblock In {\em The {M}az'ya anniversary collection, {V}ol. 2 ({R}ostock,
  1998)}, volume 110 of {\em Oper. Theory Adv. Appl.}, pages 55--67.
  Birkh\"auser, Basel, 1999.

\bibitem{DHK}
P. Dr\'{a}bek, H. P. Heinig and A. Kufner. {\em Higher-dimensional Hardy inequality, in: General Inequalities.} 7 (Oberwolfach 1995),
Internat. Ser. Numer. Math. 123, Birkhauser, Basel, 3--16, 1997.
\bibitem{D04}
B. Dyda.  A fractional order Hardy inequality. {\it Ill. J. Math.}, 48(2): 575-588, 2004.
\bibitem{DF12}
B. Dyda and R. L. Frank. Fractional Hardy–Sobolev–Maz’ya inequality for domains. {\it Stud. Math.}, 208(2):151–166, 2012.
\bibitem{EE04}
D.~E. Edmunds and W.~D. Evans.
\newblock {\em Hardy operators, function spaces and embeddings}.
\newblock Springer Monographs in Mathematics. Springer-Verlag, Berlin, 2004.
\bibitem{FS08}
R. L. Frank and R. Seiringer. Non-linear ground state representations and sharp Hardy inequalities. {\it J. Funct. Anal.}, 255(12):3407-3430, 2008.
\bibitem{FS10}
R. L. Frank and R. Seiringer. Sharp fractional Hardy inequalities in half-spaces. {\it In: Around
the Research of Vladimir Maz’ya. I. Function Spaces}, Springer, New York, 161-176, 2010.
\bibitem{FR}
V.~Fischer and M.~Ruzhansky.
\newblock {\em Quantization on nilpotent {L}ie groups}, volume 314 of {\em
  Progress in Mathematics}.
\newblock Birkh\"auser. (open access book), 2016.


\bibitem{FS}
G.~B. Folland and E.~M. Stein.
\newblock {\em Hardy spaces on homogeneous groups}, volume~28 of {\em
  Mathematical Notes}.
\newblock Princeton University Press, Princeton, N.J.; University of Tokyo
  Press, Tokyo, 1982.


\bibitem{Gar07}
D. J. H. Garling.
\newblock {\em Inequalities: a journey into linear analysis.}
\newblock Cambridge Univ. Press, 2007.
\bibitem{GKPW04}
A.~Gogatishvili, A.~Kufner, L.-E. Persson, and A.~Wedestig.
\newblock An equivalence theorem for integral conditions related to {H}ardy's
  inequality. {\it Real Anal. Exch.}, 29(2):867-880, 2004.
  \bibitem{GY00}
  N. Ghoussoub and C. Yuan. Multiple solutions for quasi-linear PDEs involving the critical
Sobolev and Hardy exponents. {\em Trans. Am. Math. Soc.}, 352(12): 5703--5743 (2000)
\newblock {\em Real Anal. Exchange}, 29(2):867--880, 2003/04.
\bibitem{GL}
N. Garofalo and E. Lanconelli. Frequency functions on the Heisenberg group,
the uncertainty principle and unique continuation. {\em Ann. Inst. Fourier (Grenoble)},
40(2):313--356, 1990.

\bibitem{Har20}
G.~H. Hardy.
\newblock Note on a theorem of Hilbert.
\newblock {\em Math. Z.}, 6(3-4):314--317, 1920.
\bibitem{HKP97}
H. P. Heinig, A. Kufner and L.-E. Persson.  On some fractional order Hardy inequalities. {\em J. Inequal. Appl.}, 1(1):25-46, 1997.
\bibitem{HHAT}
M. Hoffmann-Ostenhof, T. Hoffmann-Ostenhof, A. Laptev and J. Tidblom. Manyparticle Hardy inequalities. \newblock{\em J. Lond. Math. Soc. (2)}, 77(1):99--114, 2008.


\bibitem{KRSin}
A. Kassymov, M. Ruzhansky and D. Suragan. Reverse integral Hardy inequality on metric measure spaces.
{\em Ann. Fenn. Math.}, 47(1):39–55, 2021.
\bibitem{KRS1in}
A. Kassymov, M. Ruzhansky and D. Suragan. Hardy inequalities on metric measure spaces, III: the case $q\leq p\leq 0$ and applications. {\em Proc. R. Soc. A.,} 479(2269): 20220307, 2023.
\bibitem{KS19}
A. Kassymov and D. Suragan. Lyapunov-type inequalities for the fractional p-sub-Laplacian. {\it Adv. Oper. Theory}, 5(2):435-452, 2020.
\bibitem{KRS19}
A. Kassymov, M. Ruzhansky and D. Suragan.  Fractional logarithmic inequalities and blow-up results with logarithmic nonlinearity on homogeneous groups. {\em NoDEA-Nonlinear Differ. Equ. Appl.}, 27(1):1-19, 2020.
\bibitem{KS20}
A. Kassymov and D. Suragan. Fractional Hardy–Sobolev Inequalities and Existence Results for Fractional Sub-Laplacians. {\it J. Math. Sci.}, 250(2): 337--350, 2020.
\bibitem{KP03}
A.~Kufner and L.-E. Persson.
\newblock {\em Weighted inequalities of {H}ardy type}.
\newblock World Scientific Publishing Co., Inc., River Edge, NJ, 2003.

\bibitem{KPS17}
A.~Kufner, L.-E. Persson, and N.~Samko.
\newblock {\em Weighted inequalities of {H}ardy type}.
\newblock World Scientific Publishing Co. Pte. Ltd., Hackensack, NJ, second
  edition, 2017.

\bibitem{Mazya85}
V. G. Maz'ya. {\em Sobolev spaces}, Springer-Verlag, 1985.

\bibitem{OK90}
B.~Opic and A.~Kufner.
\newblock {\em Hardy-type inequalities}, volume 219 of {\em Pitman Research
  Notes in Mathematics Series}.
\newblock Longman Scientific \& Technical, Harlow, 1990.

\bibitem{RT15}
L. Roncal and S. Thangavelu.  Hardy's inequality for fractional powers of the sublaplacian on the Heisenberg group. {\it  Adv. Math.}, 302:106-158, 2016.
\bibitem{RV}
M.~Ruzhansky and D.~Verma. Hardy inequalities on metric measure spaces.
{\em Proc. R. Soc. A.,} 475(2223):20180310, 2019.
\bibitem{RV21}
M.~Ruzhansky and D.~Verma. Hardy inequalities on metric measure spaces, II: the case $p> q$. {\em Proc. R. Soc. A.,}  477(2250):20210136, 2021.

\bibitem{RS17}
M.~Ruzhansky and D.~Suragan.
\newblock Hardy and {R}ellich inequalities, identities, and sharp remainders on
  homogeneous groups.
\newblock {\em Adv. Math.}, 317:799--822, 2017.

\bibitem{RS18}
M.~Ruzhansky and D.~Suragan.
\newblock {\em Hardy inequalities on homogeneous groups}.
Progress in Math. Vol. 327, Birkh\"{a}user, 588 pp, 2019.


\bibitem{RY18b}
M.~Ruzhansky and N.~Yessirkegenov.
\newblock Hypoelliptic functional inequalities. 
\newblock {\em Math. Z.}, 307(22), 2024.

\bibitem{Y15}
J. Yang. Fractional Sobolev–Hardy inequality in $\mathbb{R}^{N}$. {\em Nonlinear Anal.-Theory Methods Appl.}, 119:179-185, 2015.


\end{thebibliography}
\end{document}